\documentclass[12pt]{amsart}

\usepackage{amsmath,amssymb,amsthm}
\usepackage{enumitem}

\usepackage{graphicx}
\usepackage{amsfonts}

\theoremstyle{plain}
\newtheorem{theorem}{Theorem}[section]
\newtheorem{proposition}[theorem]{Proposition}

\newtheorem{lemma}[theorem]{Lemma}
\theoremstyle{definition}
\newtheorem{definition}[theorem]{Definition}

\newtheorem{conjecture}[theorem]{Conjecture}

\makeatletter
\@addtoreset{equation}{section}
\makeatother

\makeatletter
\newcommand{\Spvek}[2][r]{%
  \gdef\@VORNE{1}
  \left(\hskip-\arraycolsep%
    \begin{array}{#1}\vekSp@lten{#2}\end{array}%
  \hskip-\arraycolsep\right)}

\def\vekSp@lten#1{\xvekSp@lten#1;vekL@stLine;}
\def\vekL@stLine{vekL@stLine}
\def\xvekSp@lten#1;{\def\temp{#1}%
  \ifx\temp\vekL@stLine
  \else
    \ifnum\@VORNE=1\gdef\@VORNE{0}
    \else\@arraycr\fi%
    #1%
    \expandafter\xvekSp@lten
  \fi}
\makeatother
\begin{document}

\title{Chebyshev polynomials on generalized Julia sets
}


\author{G{\"o}kalp Alpan}


\maketitle

\begin{abstract}
Let $(f_n)_{n=1}^\infty$ be a sequence of nonlinear polynomials satisfying some mild conditions. Furthermore, let $F_m(z)=(f_m\circ f_{m-1}\ldots \circ f_1)(z)$ and $\rho_m$ be the leading coefficient for $F_m$. It is shown that on the Julia set $J_{(f_n)}$, the Chebyshev polynomial of the degree $\deg{F_m}$ is of the form $F_m(z)/\rho_m-\tau_m$ for all $m\in\mathbb{N}$ where $\tau_m\in\mathbb{C}$. This generalizes the result obtained for autonomous Julia sets in \cite{kamo}.

\end{abstract}

\section{Introduction}
Let $(f_n)_{n=1}^\infty$ be a sequence of rational functions in $\overline{\mathbb{C}}=\mathbb{C}\cup \{\infty\}$. Let us define the associated compositions by $F_m(z):=f_m\circ\ldots f_1(z)$ for each $m\in\mathbb{N}$. Then the set of points in $\overline{\mathbb{C}}$ for which $(F_n)_{n=1}^\infty$ is normal in the sense of Montel is called the \emph{Fatou set} for $(f_n)_{n=1}^\infty$. The complement of the Fatou set is called the \emph{Julia set} for $(f_n)_{n=1}^\infty$ and is denoted by $J_{(f_n)}$. The metric considered here is the chordal metric. Julia sets corresponding to a sequence of rational functions, to our knowledge, were considered first in \cite{Fornaess}. There are several papers appeared in the literature (see e.g. \cite{Bruck1,Buger,comerford,Rugh}) which show the possibility for adapting the results on autonomous Julia sets to this more general setting with some minor changes. By an autonomous Julia set, we mean the set $J_{(f_n)}$ with $f_n(z)=f(z)$ for all $n\in\mathbb{N}$ where $f$ is a rational function. 

The Julia set $J_{(f_n)}$ is never empty provided that $\deg{f_n}\geq 2$ for all $n$. Moreover, $F_k^{-1}(F_k(J_{(f_n)}))=J_{(f_n)}$ for all positive $k$. If, in addition, we assume that $f_n=f$ for all $n$ then $f(J(f))= f^{-1}(J(f))=J{(f)}$ where $J(f):= J_{(f_n)}$. But without the last assumption, we do not have $g(J_{(f_n)})=J_{(f_n)}$ or $g^{-1}(J_{(f_n)})=J_{(f_n)}$ for a rational function $g$ with $\deg{g}\geq 2$, in general. That is the main reason why further techniques are needed in this framework.

If $f$ is a nonlinear complex polynomial then  $J(f)=\partial\{z\in\mathbb{C}:\, f^{(n)} (z)\rightarrow\infty\}$ and $J(f)$ is an infinite compact subset of $\mathbb{C}$ where $f^{(n)}$ is the $n$-th iteration of $f$.  The next result is due to Kamo-Borodin \cite{kamo}:

\begin{theorem}\label{kammo} Let $f(z)=z^m+a_{m-1}z^{m-1}+\ldots + a_0$ be a nonlinear complex polynomial and $T_k(z)$ be a Chebyshev polynomial on $J(f)$. Then $(T_k\circ f^{(n)})(z)$ is also a Chebyshev polynomial on $J(f)$ for each $n\in\mathbb{N}$. Moreover, there exists a complex number $\tau$ such that $f^{(n)}(z)-\tau$ is a Chebyshev polynomial on $J(f)$ for all $n\in\mathbb{N}$.
\end{theorem} 

Let $K\subset\mathbb{C}$ be a compact set with $\mathrm{Card}K\geq m$ for some $m\in\mathbb{N}$. Recall that, for every $n\in\mathbb{N}$ with $n\leq m$, the unique monic polynomial $P_n$ of degree $n$ satisfying $$\|P_n\|_K=\inf_{Q_{n-1}\in\mathcal{P}_{n-1}}\|z^n-Q_{n-1}(z)\|_K,$$ is called the $n$-th \emph{Chebyshev polynomial} on $K$ where $\|\cdot \|_K$ is the sup-norm on $K$ and $\mathcal{P}_{n-1}$ is the space of all polynomials of degree less than or equal to $n-1$.  

In Section 2, we review some standard facts about the generalized Julia sets and the Chebyshev polynomials. In the last section, we present a result which can be seen as a generalization of Theorem \ref{kammo}. Polynomials considered in these sections are always nonlinear complex polynomials unless stated otherwise. For a deeper discussion of Chebyshev polynomials, we refer the reader to \cite{sch,peter,sodin}. For different aspects of the theory of Julia sets, see \cite{alpgoc,Brolin,Bruck,Milnor} among others.

\section{Preliminaries}
Autonomous polynomial Julia sets enjoy plenty of nice properties. These sets are non-polar compact sets which are regular with respect to the Dirichlet problem. Moreover, there are a couple of equivalent ways to describe these sets. For further details, see \cite{Milnor}. In order to have similar features for the generalized case, we need to put some restrictions on the given polynomials. The conditions used in the following definition are from Section 4 in \cite{Bruck}. 
\begin{definition}Let $f_n(z)=\sum_{j=0}^{d_n}a_{n,j}\cdot z^j$ where $d_n\geq 2$ and $a_{n,d_n}\neq 0$ for all $n\in\mathbb{N}$. We say that $(f_n)$ is a \emph{regular polynomial sequence} if the following properties are satisfied:
\begin{itemize}
\item There exists a real number $A_1> 0$ such that $|a_{n,d_n}|\geq A_1$, for all $n\in\mathbb{N}$.
\item There exists a real number $A_2\geq 0$ such that $|a_{n,j}|\leq A_2 |a_{n,d_n}|$ for $j=0,1,\ldots, d_n-1$ and $n\in\mathbb{N}$.
\item There exists a real number $A_3$ such that $$\log{|a_{n,d_n}|}\leq A_3\cdot d_n,$$
for all $n\in\mathbb{N}$. 
\end{itemize}
\end{definition}
If $(f_n)$ is a regular polynomial sequence then we use the notation $(f_n)\in\mathcal{R}$. Here and in the sequel, $F_l(z):=(f_l\circ\ldots\circ f_1)(z)$ and $\rho_l$ is the leading coefficient of $F_l$. Let $\mathcal{A}_{(f_n)}(\infty):=\{z\in\overline{\mathbb{C}}: (F_n(z))_{n=1}^\infty \mbox{ goes locally uniformly to } \infty \}$ and $\mathcal{K}_{(f_n)}:=\{z\in\mathbb{C}: (F_n(z))_{n=1}^\infty \mbox{ is bounded}\}$. In the next theorem, we list some facts that will be necessary for the subsequent results.

\begin{theorem}\label{kar}\cite{Bruck} Let $(f_n)\in\mathcal{R}$. Then the following hold:
\begin{enumerate}[label={(\alph*})]
\item $J_{(f_n)}$ is a compact set in $\mathbb{C}$ with positive logarithmic capacity.

\item For each $R>1$ satisfying 
\begin{equation}\label{aaaa}
A_1 R\left(1-\frac{A_2}{R-1}\right)>2,
\end{equation}
we have $\mathcal{A}_{(f_n)}(\infty)=\cup_{k=1}^{\infty}{F_k}^{-1}(\triangle_R)$ and $f_n({\overline{{\triangle}_R}})\subset \triangle_R$ where $\triangle_R= \{z\in\overline{\mathbb{C}}: |z|>R\}.$ Furthermore, $\mathcal{A}_{(f_n)}(\infty)$ is a domain in $\overline{\mathbb{C}}$ containing $\triangle_R$. 

\item  $\triangle_R\subset\overline{F_k^{-1}(\triangle_R)}\subset F_{k+1}^{-1}(\triangle_R)\subset \mathcal{A}_{(f_n)}(\infty)$ for all $k\in\mathbb{N}$ and each $R>1$ satisfying \eqref{aaaa}.

\item $\partial\mathcal{A}_{(f_n)}(\infty)=J_{(f_n)}=\partial\mathcal{K}_{(f_n)}$ and $\mathcal{K}_{(f_n)}=\overline{\mathbb{C}}\setminus \mathcal{A}_{(f_n)}(\infty)$. Thus, $\mathcal{K}_{(f_n)}$ is a compact subset of $\mathbb{C}$ and $J_{(f_n)}$ has no interior points. 
\end{enumerate}
\end{theorem}
The next result is an immediate consequence of Theorem \ref{kar}.
\begin{proposition}\label{prereq}
Let $(f_n)\in\mathcal{R}$. Then $$\lim_{k\rightarrow\infty}\left(\sup_{a\in\overline{\mathbb{C}}\setminus F_k^{-1}(\triangle_R)} \mathrm{dist}(a, \mathcal{K}_{(f_n)})\right)=0,$$
where $\mathrm{dist}(\cdot)$ is the distance function and $R$ be a real number satisfying \eqref{aaaa}.
\end{proposition}
\begin{proof}

Since the Euclidean metric and the chordal metric are strongly equivalent on the compact subsets of $\mathbb{C}$, we consider here the Euclidean metric. By using parts $(b),(c)$ and $(d)$ of Theorem \ref{kar}, we have $\overline{\mathbb{C}}\setminus{F_{k+1}^{-1}(\triangle_R)}\subset\overline{\mathbb{C}}\setminus F_{k}^{-1}(\triangle_R)$ which implies that $$\displaystyle a_k:=\sup_{a\in\overline{\mathbb{C}}\setminus F_k^{-1}(\triangle_R)} \mathrm{dist}(a, \mathcal{K}_{(f_n)})$$ is a decreasing sequence.

Suppose that $a_k\rightarrow \epsilon$ as $k\rightarrow\infty$ for some $\epsilon>0$. Then, by compactness of the set $\overline{\mathbb{C}}\setminus F_k^{-1}(\triangle_R)$, there exists $b_k\in \overline{\mathbb{C}}\setminus F_k^{-1}(\triangle_R)$ for each $k$ such that $\mathrm{dist}(b_k, \mathcal{K}_{(f_n)})\geq \epsilon.$ But since $\cap_{k=1}^\infty \overline{\mathbb{C}}\setminus F_k^{-1}(\triangle_R)= \mathcal{K}_{(f_n)}$ by parts $(b)$ and $(d)$ of Theorem $\ref{kar}$, $(b_k)$ should have an accumulation point $b$ in $\mathcal{K}_{(f_n)}$ with $\mathrm{dist}(b, \mathcal{K}_{(f_n)})>\epsilon/2$ which is clearly impossible. This completes the proof.
\end{proof}
For a compact set $K\subset\mathbb{C}$, the smallest closed disc $B(a,r)$ containing $K$ is called the \emph{Chebyshev disk} for $K$. The center $a$ of this disk is called the \emph{Chebyshev center} of $K$. These concepts were crucial and widely used in the paper \cite{kamo}. The next result which is vital for the proof of Lemma \ref{lemlem} is from \cite{pako}:
\begin{theorem}\label{uyuy}
Let $L\subset\mathbb{C}$ be a compact set with $\mathrm{card} L\geq 2$ having the origin as its Chebyshev center. Let $L_p= p^{-1}(L)$ for some monic complex polynomial $p$ with $\deg{p}=n$. Then $p$ is the unique Chebyshev polynomial of degree $n$ on $L_p$. 
\end{theorem}   
\section{Results}
First, we begin with a lemma which is also interesting in its own right. 
\begin{lemma}\label{lemlem}Let $f$ and $g$ be two non-constant complex polynomials and $K$ be a compact subset of $\mathbb{C}$ with
$\mathrm{card} K\geq 2$. Furthermore, let $\alpha$ be the leading coefficient of $f$. Then the following propositions hold.
\begin{enumerate}[label={(\alph*})]
\item The Chebyshev polynomial of the degree $\deg{f}$ on the set $(g\circ f)^{-1}(K)$ is of the form $f(z)/\alpha-\tau$ where $\tau\in\mathbb{C}$. 
\item If $g$ is given as a linear combination of monomials of even degree and $K=\overline{D(0,R)}$ for some $R>0$ then the $\deg{f}$-th Chebyshev polynomial on the set $(g\circ f)^{-1}(K)$ is  $f(z)/\alpha$.
\end{enumerate}
\end{lemma}
\begin{proof} Let $K_1:=g^{-1}(K)$. Then $(g\circ f)^{-1}(K)=f^{-1}(K_1)=(f/\alpha)^{-1}(K_1/\alpha)$ where $K_1/\alpha=\{z: z=z_1/\alpha \mbox{ for some } z_1\in K_1\}$. By the fundamental theorem of algebra, $\mathrm{card} (K_1/\alpha)=\mathrm{card} K_1 \geq\mathrm{card} K$ and $K_1$ is compact by the continuity of $g(z)$. The set $K_1/\alpha$ is also compact since the compactness of a set is preserved under a linear transformation. Let $\tau$ be the Chebyshev center for $K_1/\alpha$. Then $K_1/\alpha-\tau$ is a compact set with the Chebyshev center as the origin where $K_1/\alpha-\tau:= \{z: z=z_1-\tau \mbox{ for some } z_1\in K_1/\alpha\}$. Note that, $\mathrm{card} (K_1/\alpha-\tau)=\mathrm{card} (K_1/\alpha)$ and $(f/a)^{-1}(K_1/\alpha)=(f/\alpha-\tau)^{-1}(K_1/\alpha-\tau)$. Using Theorem \ref{uyuy}, for $p(z)=f(z)/\alpha-\tau$ and $L=K_1/\alpha-\tau$, we see that $p(z)$ is the $\deg{f}$-th Chebyshev polynomial on the set $L_p=(g\circ f)^{-1}(K)$. This proves the first part of the theorem.

Suppose further that $g(z)=\sum_{j=0}^n a_j\cdot z^{2j}$ for some $n\geq 1$ and $(a_0, \ldots, a_n)\in \mathbb{C}^{n+1}$ with $a_n\neq 0$. Let $K=\overline{D(0,R)}$ for some $R>0$. Then the Chebyshev center for $K_1/\alpha=g^{-1}(K)/\alpha=g^{-1}(D(0,R))/\alpha$ is the origin since $g(z)/\alpha=g(-z)/\alpha$ for all $z\in\mathbb{C}$. Thus, $f(z)/\alpha$ is the $\deg{f}$-th Chebyshev polynomial for $(g\circ f)^{-1}(K)$ under these extra assumptions.
\end{proof}
The next theorem shows that it is possible to obtain similar results to Theorem \ref{kammo} in a richer setting.
\begin{theorem}\label{tete}
Let $(f_n)\in\mathcal{R}$. Then the following hold:
\begin{enumerate}[label={(\alph*})]
\item  For each $m\in\mathbb{N}$, the $\deg{F_m}$-th Chebyshev polynomial on $J_{(f_n)}$ is of the form $F_m(z)/\rho_m-\tau_m$ where $\tau_m\in\mathbb{C}$.
\item If, in addition, each $f_n$ is given as a linear combination of monomials of even degree then $F_m(z)/\rho_m$ is the $\deg{F_m}$-th Chebyshev polynomial on $J_{(f_n)}$ for all $m$.
\end{enumerate}
\end{theorem}
\begin{proof}
Let $m\in\mathbb{N}$ be given and $R>1$ satisfy \eqref{aaaa}. For each natural number $l>m$, define $g_l:=f_l\circ\ldots\circ f_{m+1}$. Then $F_l=g_l\circ F_m$ for each such $l$. Using part $(a)$ of Lemma \ref{lemlem} for $g=g_l$, $f=F_m$ and $K=\overline{D(0,R)}$, we see that the $(d_1\cdots d_m)$-th Chebyshev polynomial on $(g_l\circ F_m)^{-1}(\overline{D(0,R)})$ is of the form $F_m(z)/\rho_m-\tau_l$ where $\tau_l\in\mathbb{C}$. Let $C_l:=||F_m(z)/\rho_m-\tau_l||_{(g_l\circ F_m)^{-1}(K)}$. Note that, by part $(c)$ of Theorem \ref{kar}, $F_s^{-1}(\overline{D(0,R)})\subset F_t^{-1}(\overline{D(0,R)})$ provided that $s<t$. This implies that $(C_j)_{j=m+1}^\infty$ is a decreasing sequence of positive numbers and hence has a limit $C$. The last follows from the observation that the norms of the Chebyshev polynomials of the same degree on a nested sequence of compact sets constitute a decreasing sequence. 

Let $P_{d_1\cdots d_m}(z)=\sum_{j=0}^{d_1\cdots d_m}a_j z^j$ be the $(d_1\cdots d_m)$-th Chebyshev polynomial on $\mathcal{K}_{(f_n)}$. Since $\mathcal{K}_{(f_n)}\subset (g_l\circ F_m)^{-1}(\overline{D(0,R)})$ for each $l$, $||P_{d_1\cdots d_m}||_{\mathcal{K}_{(f_n)}}:=C_0\leq C$. Suppose that $C_0<C$.

Let $\epsilon= \min\{(C-C_0)/2,1\}$ and $$\delta=\frac{\epsilon}{\max\{|a_1|, |a_2|,\ldots,|a_{d_1\cdots d_m}|\}(8R)^{d_1\cdots d_m} (d_1\cdots d_m)^2}.$$
By Proposition \ref{prereq}, there exists a real number $N_0>m$ such that $N>N_0$ with $N\in\mathbb{N}$ implies that $$\sup_{z\in\overline{\mathbb{C}}\setminus F_N^{-1}(\triangle_R)} \mathrm{dist}(z, \mathcal{K}_{(f_n)})<\delta.$$ Therefore, for any $z\in F_{N_0+1}^{-1}(\overline{D(0,R)})$, there exists a number $z^\prime\in\mathcal{K}_{(f_n)}$ with $|z-z^\prime|<\delta$. Hence, for each $z\in F_{N_0+1}^{-1}(\overline{D(0,R)})$,
\begin{equation*}
|P_{d_1\cdots d_m}(z)|<|P_{d_1\cdots d_m}(z^\prime)|+\frac{\epsilon}{2}<C\leq\bigg|\bigg| \frac{F_m(z)}{\rho_m}-\tau_{N_0+1}\bigg|\bigg|_{F_{N_0+1}^{-1}(\overline{D(0,R)})},
\end{equation*}
which contradicts with the fact that $F_m(z)/{\rho_m}+\tau_{N_0+1}$ is the $(d_1\cdots d_m)$-th Chebyshev polynomial on $F_{N_0+1}^{-1}(\overline{D(0,R)})$. Thus $C_0=C$. 

Since $(C_j)_{j=m+1}^\infty$ is decreasing we have $|\tau_l|<2C_m+|\tau_m|$ provided that $l>m$. Thus $(\tau_l)_{l=m+1}^\infty$ has at least one convergent subsequence $(\tau_{l_k})_{k=1}^{\infty}$ with a limit $\tau_m$. Therefore, 
\begin{equation}\label{huhu}
C\leq \lim_{k\rightarrow \infty}\bigg|\bigg|\frac{F_m(z)}{\rho_m}-\tau_m \bigg|\bigg|_{F_{l_k}^{-1}(\overline{D(0,R)})}\leq \lim_{k\rightarrow \infty} (C_{l_k}+|\tau_{l_k}-\tau_m|)=C.
\end{equation}
By uniqueness of the Chebyshev polynomials and \eqref{huhu}, $F_m(z)/\rho_m-\tau_m$ is $(d_1\cdots d_m)$-th Chebyshev polynomial on $\mathcal{K}_{(f_n)}$. By the maximum principle, for any polynomial $Q$, we have $$||Q||_{\mathcal{K}_{(f_n)}}=||Q||_{\partial\mathcal{K}_{(f_n)}}=||Q||_{J_{(f_n)}}.$$
Hence the Chebyshev polynomials on $\mathcal{K}_{(f_n)}$ and $J_{(f_n)}$ should coincide. This proves the first assertion.

Suppose further that, the assumption given in part $(b)$ is satisfied. Then by part $(b)$ of Lemma \ref{lemlem}, for $g=g_l$, $f=F_m$ and $K=\overline{D(0,R)}$, the $(d_1\cdots d_m)$-th Chebyshev polynomial on $(g_l\circ F_m)^{-1}(\overline{D(0,R)})$ is of the form $F_m(z)/\rho_m-\tau_l$ where $\tau_l=0$ for $l>m$. Any subsequence of $(\tau_l)_{l=m+1}^\infty$ converges to $0$. Thus, arguing as above, we can reach the conclusion that $F_m(z)/\rho_m$ is the $(d_1\cdots d_m)$-th Chebyshev polynomial for $J_{(f_n)}$ under this extra assumption. This completes the proof.
\end{proof}

This theorem gives the total description of $2^n$ degree Chebyshev polynomials for the most studied case, i.e., $f_n(z)=z^2+c_n$ with  $c_n\in\mathbb{C}$ for all $n$.  If $(c_n)_{n=1}^\infty$ is bounded then the logarithmic capacity of $J_{(f_n)}$ is $1$. Moreover, if $|c_n|\leq 1/4$ for all $n$, $J_{(f_n)}$ is connected. If $|c_n|<c<1/4$ for some positive $c$, then $J_{(f_n)}$ is quasicircle. See \cite{Bruck1}, for the definition of a quasicircle and proofs of the above facts.

For a non-polar infinite compact set $K\in\mathbb{C}$, let us define the sequence $(W_n(K))_{n=1}^\infty$ by $W_n(K)=||P_n||/(\mathrm{Cap}(K))^n$ for all $n\in\mathbb{N}$. There are recent studies on the asymptotic behaviour of these sets on several occasions. See e.g. \cite{andr,gonchat,totikvar}. 

In order $W_n(K)$ to be bounded, there are sufficient conditions given in terms of the smoothness of the outer boundary of $K$ in \cite{andr,totikvar}. There is also an old and open question proposed by Ch.Pommerenke which is in the inverse direction: Find (if it is possible) a continuum $K$ with $\mathrm{Cap}(K)=1$ such that $(W_n(K))_{n=1}^\infty$ is unbounded. To answer this question positively, it is very natural to consider a continuum with a nonrectifiable outer boundary. Thus, we make the following conjecture:
\begin{conjecture}\label{concon}
Let $f(z)=z^2+1/4$. Then, $(W_n(J(f))_{n=1}^\infty$ is unbounded. 
\end{conjecture}
As discussed in \cite{Bruck}, for $f(z)=z^2+1/4$, $J(f)$ has dimension $2$ and in this case $J(f)$ is not a quasicircle. Hence, Theorem 2 of \cite{andr} is not applicable since it requires even stronger assumptions on the outer boundary. But, unfortunately, proving Conjecture \ref{concon} is not an easy task. Theorem \ref{kammo} is not sufficient to find all Chebyshev polynomials on $J(f)$ explicitly.  Nonetheless, one can claim, further, that for $f_n(z)= z^2+1/4-\epsilon_n$, with a fastly decreasing positive sequence $(\epsilon_n)_{n=1}^\infty$ (for example if the terms satisfy $\sum_{n=1}^\infty \sqrt{\epsilon_n}<\infty$ with $0<1/4<\epsilon_n$), the sequence $(W_n(J_{(f_n)}))_{n=1}^\infty$ is also unbounded.


\begin{thebibliography}{15}
 \bibitem{alpgoc}Alpan, G., Goncharov, A.: {Orthogonal polynomials on generalized Julia sets}. Manuscript submitted for publication.
 
\bibitem{andr}Andrievskii, V.V.: {Chebyshev Polynomials on a System of Continua}. Constr. Approx. doi: 10.1007/s00365-015-9280-8

\bibitem{Brolin}Brolin, H.: {Invariant sets under iteration of rational functions}. Ark. Mat. \textbf{6}(2), 103--144 (1965)

\bibitem{Bruck1}Br\"{u}ck, R.: {Geometric properties of Julia sets of the composition of polynomials of the form $z^2 +c_n$}. Pac. J. Math. \textbf{198}, 347--372 (2001)

\bibitem{Bruck} Br\"{u}ck, R., B\"{u}ger, M.: {Generalized Iteration}. Comput. Methods Funct. Theory \textbf{3}, 201--252 (2003) 

\bibitem{Buger} B\"{u}ger, M.: {Self-similarity of Julia sets of the composition of polynomials}. Ergodic Theory Dyn. Syst. \textbf{17}, 1289--1297 (1997)

\bibitem{comerford} Comerford, M.: {Hyperbolic non-autonomous Julia sets}. Ergodic Theory Dyn. Syst. \textbf{26}, 353--377 (2006) 

\bibitem{Fornaess}Forn\ae ss, J.E., Sibony, N.: {Random iterations of rational functions}. Ergodic Theory Dyn. Syst. \textbf{11}, 687--708 (1991)

\bibitem{gonchat}Goncharov, A., Hatino\u{g}lu, B.: {Widom Factors}. Potential Anal. \textbf{42}, 671-680 (2015)

\bibitem{kamo} Kamo, S.O., Borodin, P.A.: {Chebyshev polynomials for Julia sets}. Mosc. Univ. Math. Bull. \textbf{49}, 44-45 (1994)

\bibitem{Milnor} Milnor, J.: {Dynamics in one complex variables}. Princeton Universty Press, Annals of Mathematics Studies,
\textbf{160}, Princeton University Press, Princeton, NJ, (2006)

\bibitem{pako} Ostrovskii, I.V., Pakovitch, F., Zaidenberg, M.G.: {A remark on complex polynomials of least deviation}. Internat. Math. Res. Notices. \textbf{14}, 699--703 (1996)

\bibitem{sch} Peherstorfer, F., Schiefermayr, K.: {Description of extremal polynomials on several intervals and their computation I, II}. Acta Math. Hungar. \textbf{83}, 27--58, 59--83 (1999)

\bibitem{peter} Peherstorfer, F., Steinbauer, R.: {Orthogonal and $L_q$-extremal polynomials on inverse images of polynomial mappings}. J. Comput. Appl. Math. \textbf{127}, 297--315 (2001) 

\bibitem{pom} Pommerenke, Ch.: {Problems in Complex Function Theory}. Bull. London Math. Soc. \textbf{4}, 354-366 (1972)

\bibitem{Rugh} Rugh, H.H.: {On the dimensions of conformal repellers.
Randomness and parameter dependency}. Ann. Math. \textbf{168}(3), 695--748 (2008)

\bibitem{sodin}Sodin, M., Yuditskii, P.: {Functions deviating least from zero on closed subsets of the real axis}. St. Petersbg. Math. J. \textbf{4}, 201--249 (1993)

\bibitem{totikvar}Totik, V., Varga, T.: {Chebyshev and fast decreasing polynomials}. Proc. London Math. Soc. doi:10.1112/plms/pdv014

\end{thebibliography}
\end{document}